\documentclass[12pt, a4paper]{article}

\usepackage{amsfonts}
\usepackage{amssymb}
\usepackage{amscd}
\usepackage{mathrsfs}
\usepackage{amsmath,amssymb, amsthm}
\usepackage{graphicx}
\usepackage{color}
\usepackage{authblk}
\usepackage{indentfirst,latexsym,amssymb}
\usepackage[bookmarks=false]{hyperref}

\newtheorem{theorem}{Theorem}[section]
\newtheorem{pro}[theorem]{Proposition}

\newtheorem{lem}[theorem]{Lemma}

\newtheorem{coro}[theorem]{Corollary}

\theoremstyle{definition}

\newtheorem{exam}[theorem]{Example}

\def\PSL{\hbox{\rm PSL}}

\def\Cay{\hbox{\rm Cay}}

\def\PGL{\hbox{\rm PGL}}

\def\ZZZ{\mathbb{Z}}

\def\Ga{\Gamma}

\usepackage{enumerate}
\long\def\delete#1{}

\usepackage{xcolor}
\usepackage[normalem]{ulem}

\input amssym.def
\input amssym.tex

\newcommand{\be}{\begin{equation}}
\newcommand{\ee}{\end{equation}}
\newcommand{\bea}{\begin{eqnarray}}
\newcommand{\eea}{\end{eqnarray}}
\newcommand{\bean}{\begin{eqnarray*}}
\newcommand{\eean}{\end{eqnarray*}}

\begin{document}

\title{Characterizing subgroup perfect codes by $2$-subgroups}
\author{Junyang Zhang}
\affil{{\small School of Mathematical Sciences, Chongqing Normal University\\ Chongqing 401331, People's Republic of China}}

\date{}

\maketitle

\footnotetext{E-mail addresses: jyzhang@cqnu.edu.cn (Junyang Zhang)}
\begin{abstract}
A perfect code in a graph $\Ga$ is a subset $C$ of $V(\Ga)$ such that no two vertices in $C$ are adjacent and every vertex in $V(\Ga)\setminus C$ is adjacent to exactly one vertex in $C$. Let $G$ be a finite group and $C$ a subset of $G$. Then $C$ is said to be a perfect code of $G$ if there exists a Cayley graph of $G$ admiting $C$ as a perfect code. It is proved that a subgroup $H$ of $G$ is a perfect code of $G$ if and only if a Sylow $2$-subgroup of $H$ is a perfect code of $G$.  This result provides a way to simplify the study of subgroup perfect codes of general groups to the study of subgroup perfect codes of $2$-groups. As an application, a criterion for determining subgroup perfect codes of projective special linear groups $\PSL(2,q)$ is given.

\medskip
{\em Keywords:} Cayley graph; perfect code; subgroup perfect code; projective special linear group

\medskip
{\em AMS subject classifications (2020):} 05C25, 05C69, 94B99
\end{abstract}

\section{Introduction}
\label{sec:intro}
All groups considered in the paper are finite groups with identity element denoted by $1$, and all graphs considered are finite, undirected and simple.
For a graph $\Gamma$, we use  $V(\Gamma)$ and $E(\Gamma)$ to denote its vertex set and edge set respectively. The \emph{distance} in $\Ga$ between two vertices is the length of a shortest path between the two vertices or $\infty$ if there is no path in $\Gamma$ joining them. Let $r$ be a positive integer. A subset $C$ of $V(\Gamma)$ is called \cite{Big, Kr86} a \emph{perfect $r$-error-correcting code} (or \emph{perfect $r$-code} for short) in $\Gamma$ if every vertex of $\Gamma$ is at distance no more than $r$ to exactly one vertex in $C$.  A perfect $1$-code is usually called a {\em perfect code}. Equivalently, a subset $C$ of $V(\Gamma)$ is a perfect code in $\Gamma$ if $C$ is an independent set of $\Gamma$ and every vertex in $V(\Gamma) \setminus C$ is adjacent to exactly one vertex in $C$. A perfect code in a graph is also called an efficient dominating set \cite{DeS} or independent perfect dominating set \cite{Le} of the graph.

The notion of perfect $r$-codes in graphs was firstly introduced by Biggs \cite{Big, Kr86} as a generalization of the notions of perfect $r$-codes under the Hamming metric and Lee metric. Recall that in coding theory the Hamming distance between words of length $n$ over an alphabet of size $m \ge 2$ is precisely the graph distance in the Hamming graph $H(n, m)$ \cite{MS77}. Therefore perfect $r$-codes in $H(n,m)$ are exactly those in the classical setting under the Hamming metric. Similarly, the Lee distance \cite{HK18} between words of length $n$ over an alphabet of size $m \ge 3$ is precisely the graph distance in the Cartesian product $L(n,m)$ of $n$ copies of the cycle of length $m$ and therefore perfect $r$-codes in $L(n,m)$ are exactly those in the classical setting under the Lee metric.

Let $G$ be a group and $S$ a subset of $G$ satisfying $S^{-1}:=\{x^{-1}\mid x \in S\} = S$ and $1\notin S$. The {\em Cayley graph} ${\rm Cay}(G, S)$ of $G$ with \emph{connection set} $S$ is defined as the graph with vertex set $G$ such that two elements $x,y$ of $G$ are adjacent if and only if $yx^{-1}\in S$.
Note that $H(n,m)$ and $L(n,m)$ are Cayley graphs of the additive group $\ZZZ_m^n$ with connection sets $S_H$ and $S_L$, respectively, where $S_H$ consists of all elements of $\ZZZ_m^n$ with precisely one nonzero coordinate, and $S_L$ consists of all elements of $\ZZZ_m^n$ such that exactly one coordinate is $\pm 1\pmod m$ and all other coordinates are zero.

Perfect codes in Cayley graphs have been extensively studied in the literature, see \cite[Section 1]{HXZ18} for a brief survey and \cite{DSLW16, FHZ, Ta13, Z15} for a few recent papers. In particular, perfect codes in Cayley graphs which are subgroups of the underlying groups are especially interesting since they are generalizations of perfect linear codes \cite{Va75} in the classical setting. In \cite{HXZ18}, Huang et al introduced the following concepts: A subset $C$ of a group $G$ is called a {\em perfect code} of $G$ if it is a perfect code of some Cayley graph of $G$; if further $C$ is a subgroup of $G$, then $C$ is called a {\em subgroup perfect code} of $G$.

In \cite{HXZ18}, some interesting results on normal subgroups of a group to be perfect codes were obtained. Subsequently, these results are extended to general subgroups of a group \cite{MWWZ20,CWX2020,ZZ2021}. Let $G$ be a group and $H$ a subgroup of $G$. In \cite{ZZ2021,ZZ2022}, the author and Zhou proved that $H$ is a perfect code of $G$ if there exists a Sylow $2$-subgroup of $H$ which is a perfect code of $G$. They also proved that for a metabelian group $G$, a normal subgroup $H$ of $G$ is a perfect code of $G$ if and only if a Sylow $2$-subgroup of $H$ is a perfect code of $G$. This result was recently generalized to all groups by Khaefi et al \cite{KAK2022}. In this paper, we show that the restriction on the normality of $H$ can also be removed. Actually, we prove the following result.
\begin{theorem}
\label{sl2}
Let $G$ be a finite group and $H$ a subgroup of $G$. Then the following statements are equivalent:
\begin{enumerate}
  \item every  Sylow $2$-subgroup of $H$ is a perfect code of $G$;
  \item $H$ has a Sylow $2$-subgroup which is a perfect code of $G$;
  \item $H$ is a perfect code of $G$.
\end{enumerate}
\end{theorem}
In order to reduce the problem of determining subgroup perfect codes of general groups to that of $2$-groups, we prove a result by using Theorem \ref{sl2} as follows.
\begin{theorem}
\label{NGQ}
Let $G$ be a finite group and $H$ a subgroup of $G$. Let $Q$ be a Sylow $2$-subgroup of $H$ and $P$ a Sylow $2$-subgroup of $N_{G}(Q)$. Then $H$ is a perfect code of $G$ if and only if $Q$ is a perfect code of $P$.
\end{theorem}
There are four sections in this paper. After this introduction section, in Section 2 we fix the notations for the paper and list some known results for later use. In Section 3, we prove  Theorems \ref{sl2} and \ref{NGQ}, and we also deduce a few other results base on Theorem \ref{NGQ}. As an application of our main reults, in Section 4 we give a criterion  for determining subgroup perfect codes of the $2$-dimensional projective special linear groups $\PSL(2,q)$ where $q$ is a prime power.

\section{Preliminaries}
\label{sec:pre}
We at first fix some notations and terminologies.
For a finite set $S$, we use $|S|$ to denote the number of elements contained in $S$. Let $G$ be a group. For a subgroup $H$ of $G$, we use $|G:H|$ and $N_{G}(H)$ to denote the index of $H$ in $G$ and the normalizer of $H$ in $G$ respectively. For any subset $X$ of $G$, we use $\langle X\rangle$ to denote the subgroup of $G$ generated by $X$. An element of $G$ is said to be an \emph{involution} if it is of order $2$. A subgroup $P$ of $G$ with order a power of a prime $p$ is called a \emph{$p$-subgroup} of $G$. If further $|G:P|$ is not divisible by $p$, then $P$ is called a \emph{Sylow $p$-subgroup} of $G$.

\medskip
Now we list some known results for later use. The first two lemmas seems some what trivial and can be found in \cite{ZZ2021}.

\begin{lem}[\cite{ZZ2021}]
\label{sub}
Let $G$ be a group and $H$ a subgroup of $G$. Then $H$ is a perfect code of $G$ if and only if it is a perfect code of any subgroup of $G$ which contains $H$.
\end{lem}
\begin{lem}[\cite{ZZ2021}]
\label{conjugate}
Let $G$ be a group and $H$ a subgroup of $G$. If $H$ is a perfect code of $G$, then for any $g \in G$, $g^{-1}Hg$ is a perfect code of $G$.  More specifically, if $H$ is a perfect code in $\Cay(G, S)$ for some connection set $S$ of $G$, then $g^{-1}Hg$ is a perfect code in $\Cay(G, g^{-1}Sg)$.
\end{lem}
The following lemma is one of the main results in \cite{HXZ18}.
\begin{lem}[\cite{HXZ18}]
\label{HXZ}
Let $G$ be a group and $H$ a normal subgroup of $G$. Then $H$ is a perfect code of $G$ if and only if for all $x\in G$, $x^{2}\in H$ implies $(xh)^{2}=1$ for some $h\in H$.
\end{lem}
The next three lemmas are from \cite{ZZ2022}.
\begin{lem}[\cite{ZZ2022}]
\label{basic}
Let $G$ be a group and $H$ a subgroup of $G$. Then $H$ is not a perfect code of $G$  if and only if there exists a double coset $D = HxH$ with $D=D^{-1}$ having an odd number of left cosets of $H$ in $G$ and containing no involution. In particular, if $H$ is not a perfect code of $G$, then there exists a $2$-element
$x\in G\setminus H$ such that $x^2\in H$, $|H : H\cap xHx^{-1}|$ is odd, and $HxH$ contains no involution.
\end{lem}

\begin{lem}[\cite{ZZ2022}]
\label{normal}
Let $G$ be a group and $H$ a subgroup of $G$. Suppose that either $H$ is a $2$-group or at least one of $|H|$ and $|G:H|$ is odd. Then $H$ is a perfect code of $G$ if and only if $H$ is a perfect code of $N_{G}(H)$.
\end{lem}

\begin{lem}[\cite{ZZ2022}]
\label{equivalent}
Let $G$ be a group and $H$ a subgroup of $G$. Suppose that either $H$ is a $2$-group or at least one of $|H|$ and $|G:H|$ is odd. Then $H$ is a perfect code of $G$ if and only if for any $x\in N_{G}(H)$, $x^{2}\in H$ implies $(xh)^{2}=1$ for some $h\in H$.
\end{lem}
Note that some corrigenda of the results in \cite{ZZ2021} were published in \cite{ZZ2022} and the statements of following two results from \cite{ZZ2021} listed below are unchanged and correct.
\begin{lem}[\cite{ZZ2021}]
\label{odd}
Let $G$ be a group and $H$ a subgroup of $G$. If either the order of $H$ is odd or the index of $H$ in $G$ is odd, then $H$ is a perfect code of $G$.
\end{lem}
\begin{lem}[\cite{ZZ2021}]
\label{ns}
Let $G$ be a group and $H$ a subgroup of $G$. If there exists a Sylow $2$-subgroup of $H$ which is a perfect code of $G$, then $H$ is a perfect code of $G$.
\end{lem}
 The following lemma is a part of the famous Sylow's Theorem (see \cite[3.2.3]{KS2004}).
 \begin{lem}
\label{sylow}
Let $p$ be a prime divisor of $|G|$. Then every $p$-subgroup of $G$ is contained in a Sylow $p$-subgroup of $G$ and all Sylow $p$-subgroups of $G$ are conjugate in $G$.
\end{lem}

\section{Main results}
\label{sec:main}
In this Section, we present the proofs of Theorems \ref{sl2} and \ref{NGQ}. We also deduce a few Corollaries of Theorem \ref{NGQ}.

\medskip
Let us prove Theorem \ref{sl2} first.
\begin{proof}[Proof of Theorem \ref{sl2}]
It is obvious that (ii) is true if (i) is true. Lemma \ref{ns} implies (ii)$\Rightarrow$(iii). In what follows we deduce (iii)$\Rightarrow$(i).

Suppose that $H$ is a perfect code of $G$. Let $Q$ be an arbitrary Sylow 2-subgroup of $H$. It suffices to show that $Q$ is a perfect code of $G$. If $|N_{G}(Q):Q|$ is odd, then Lemma \ref{odd} implies that $Q$ is a perfect code of $N_{G}(Q)$ and it follows from Lemma \ref{normal} that $Q$ is a perfect code of $G$ . In what follows, we assume that  $|N_{G}(Q):Q|$ is even.
Consider an arbitrary element $x\in N_{G}(Q)\setminus Q$ with $x^{2}\in Q$. Set $P=\langle Q,x\rangle$. Then $Q$ is a normal subgroup of $P$ of index $2$. In particular, $P=Q\cup xQ$. Since $Q$ is a Sylow 2-subgroup of $H$, we get $x\notin H$ and $x^2\in H$. Therefore $(HxH)^{-1}=HxH$. Furthermore, $|H:H\cap xHx^{-1}|$ is odd as $Q=xQx^{-1}\subseteq H\cap xHx^{-1}$. It follows that $HxH$ is a union of an odd number of left cosets of $H$ in $G$. Since $H$ is a perfect code of $G$, it follows from Lemma \ref{basic} that $HxH$ contains an involution, say $h_1xh_2$. Let $h=h_2h_1$. Since $xh=h_1^{-1}(h_1xh_2)h_1$, $xh$ is an involution in $xH$. Set $L=\langle Q,x,h\rangle$. Since $xhx^{-1}=(xh)^2h^{-1}x^{-2}=h^{-1}x^{-2}\in \langle Q,h\rangle$ and $xQx^{-1}=Q$, we conclude that $\langle Q,h\rangle$ is a normal subgroup of $L$. Since $x\notin\langle Q,h\rangle$ and $x^2\in\langle Q,h\rangle$, we obtain that $\langle Q,h\rangle$ is of index $2$ in $L$. Therefore $|L|/|P|=|\langle Q,h\rangle|/|Q|$. Since $Q$ is a Sylow $2$-subgroup of $H$ and $\langle Q,h\rangle$ is a subgroup of $H$ containing $Q$, we have that $|\langle Q,h\rangle|/|Q|$ is odd. Therefore $|L|/|P|$ is odd and it follows that $P$ is a Sylow $2$-subgroup of $L$. Since $xh$ is an involution contained in $L$, there exists a Sylow $2$-subgroup $R$ of $L$ containing $xh$. By Lemma \ref{sylow}, the Sylow $2$-subgroups of $L$ are conjugate in $L$. Therefore there exists an element $b\in L$ such that $bRb^{-1}=P$. It follows that $bxhb^{-1}\in P$. Since $x\notin H$ and $b,h\in H$, we get $bxhb^{-1}\notin H$. Thus $bxhb^{-1}\notin Q$. It follows that $bxhb^{-1}\in xQ$ as $P=Q\cup xQ$ and $bxhb^{-1}\in P$. Since $xh$ is an involution, we have that $bxhb^{-1}$ is an involution. Now we have proved that $xQ$ contains an involution. By Corollary \ref{equivalent}, $Q$ is a perfect code of $G$.
\end{proof}

Now we prove Theorem \ref{NGQ}.
\begin{proof}[Proof of Theorem \ref{NGQ}]
$\Rightarrow$) Suppose that $H$ is a perfect code of $G$. By Theorem \ref{sl2}, we have that $Q$ is a perfect code of $G$ and it follows from Lemma \ref{sub} that $Q$ is a perfect code of $P$.

$\Leftarrow$) Suppose that $Q$ is a perfect code of $P$.
Consider an arbitrary element $x\in N_{G}(Q)$ with $x^{2}\in Q$. Then $\langle Q,x\rangle$ is a $2$-subgroup of $N_{G}(Q)$. Since $P$ is a Sylow $2$-subgroup of $N_{G}(Q)$, it follows from Lemma \ref{sylow} that $\langle Q,x\rangle$ is contained in $b^{-1}Pb$ for some $b\in N_{G}(Q)$. Thus $bxb^{-1}\in P$ and $(bxb^{-1})^{2}\in Q$.
Since $Q$ is normal in $P$ and a perfect code of $P$, by Lemma \ref{HXZ} we obtain $(bxb^{-1}c)^{2}=1$ for some $c\in Q$. It follows that $b^{-1}cb\in Q$ and $(xb^{-1}cb)^{2}=1$. By Lemma \ref{equivalent}, we have that $Q$ is a perfect code of $G$. Recall that $Q$ is a Sylow $2$-subgroup of $H$. By Theorem \ref{sl2},  $H$ is a perfect code of $G$.
\end{proof}
In what follows, we give three corollaries of Theorem \ref{NGQ}. These corollaries involve special groups $G$ or special subgroups $H$.
\begin{coro}
\label{qnormal}
Let $G$ be a group and $H$ a subgroup of $G$. Let $Q$ be a Sylow $2$-subgroup of $H$ and $P$ a Sylow $2$-subgroup of $G$ containing $Q$. If $P\cap N_{G}(Q)$ is of odd index in $N_{G}(Q)$, then $H$ is a perfect code of $G$ if and only if $Q$ is a perfect code of $P$.
\end{coro}
\begin{proof}
Note that $N_{P}(Q)=P\cap N_{G}(Q)$. Since $P\cap N_{G}(Q)$  is of odd index in $N_{G}(Q)$, we have that
$N_{P}(Q)$ is a Sylow $2$-subgroup of $N_{G}(Q)$. By Lemma \ref{normal}, $Q$ is a perfect code of $P$ if and only if $Q$ is a perfect code of $N_{P}(Q)$. Together with Theorem \ref{NGQ}, we conclude that $H$ is a perfect code of $G$ if and only if $Q$ is a perfect code of $P$.
\end{proof}
\begin{coro}
\label{ns2}
Let $G$ be a group having an normal Sylow $2$-subgroup $P$. Let $H$ be a subgroup of $G$ and $Q$ a Sylow $2$-subgroup of $H$. Then $H$ is a perfect code of $G$ if and only if $Q$ is a perfect code of $P$.
 \end{coro}
\begin{proof}
Since the Sylow $2$-subgroup $P$ is normal in $G$, we have that $P$ is the unique Sylow $2$-subgroup of $G$. Therefore every Sylow $2$-subgroup of $N_{G}(Q)$ is contained in $P$. This implies that $Q$ is a subgroup of $P$ and
$N_{P}(Q)$ is the unique Sylow $2$-subgroup of $N_{G}(Q)$. By Lemma \ref{normal}, $Q$ is a perfect code of $P$ if and only if $Q$ is a perfect code of $N_{P}(Q)$. It follows from Theorem \ref{NGQ} that $H$ is a perfect code of $G$ if and only if $Q$ is a perfect code of $P$.
\end{proof}

\begin{coro}
\label{abelianq}
Let $G$ be a group having an abelian Sylow $2$-subgroup and $H$ a subgroup of $G$. Let $Q$ be a Sylow $2$-subgroup of $H$ and $P$ a Sylow $2$-subgroup of $G$ containing $Q$. Then $H$ is a perfect code of $G$ if and only if $Q$ is a perfect code of $P$.
 \end{coro}
\begin{proof}
By Lemma \ref{sylow}, all Sylow $2$-subgroup of $G$ are conjugate in $G$. Since $G$ has an abelian Sylow $2$-subgroup and $P$ is a Sylow $2$-subgroup of $G$, we have that $P$ is abelian and therefore $Q$ is normal in $P$. Thus $P$ is a Sylow $2$-subgroup of $N_{G}(Q)$. By Theorem \ref{NGQ}, we obtain that $H$ is a perfect code of $G$ if and only if $Q$ is a perfect code of $P$.
\end{proof}
As a contrast to Corollary \ref{abelianq}, we introduce the following proposition of which the statement is an equivalent expression of a theorem \cite[Theorem 1.1]{MWWZ20} of Ma et al. Note that a $2$-group has no element of order $4$ if and only if it is elementary abelian.
\begin{pro}
\label{ele}
Let $G$ be a group. Then every subgroup of $G$ is a perfect code if and only if $G$ has an elementary abelian Sylow $2$-subgroup.
 \end{pro}
\begin{proof}
$\Rightarrow$) If $G$ has a Sylow $2$-subgroup $P$ which is not an elementary abelian group, then $P$ contains an element $z$ of order $4$. In particular, $z\in N_{G}(\langle z^2\rangle)$ and $z\langle z^2\rangle$ contain no involution. By Lemma \ref{equivalent}, $\langle z^2\rangle$ is not a perfect code of $G$. Therefore, if every subgroup of $G$ is a perfect code, then every Sylow $2$-subgroup of $G$ is elementary abelian.

$\Leftarrow$)
If $G$ has an elementary abelian Sylow $2$-subgroup, then it follows from Lemma \ref{sylow} that every Sylow $2$-subgroup of $G$ is elementary abelian. Let $H$ be a subgroup of $G$. Let $Q$ be a Sylow $2$-subgroup of $H$ and $P$ a Sylow $2$-subgroup of $N_G(Q)$. Then $P$ is an elementary abelian group containing $Q$ and therefore $Q$ is a perfect code of $P$. By Theorem \ref{NGQ}, we obtain that $H$ is a perfect code of $G$.
\end{proof}

Let $Q$ be a Sylow $2$-subgroup of $H$ which is contained in a Sylow $2$-subgroup $P$ of $G$. In general, $Q$ being a perfect code of $P$ does not guarantee that $H$ is a perfect code of $G$. See the following example.
\begin{exam}
Let $G=S_6$, $H=\langle (12)(35),(345)\rangle$, $P=\langle (12),(35),(3456)\rangle$ and $Q=\langle (12)(35)\rangle$. Then $Q$ is a Sylow $2$-subgroup of $H$ and $P$ is a Sylow $2$-subgroup of $G$ containing $Q$. It is obvious that $(1325)\in N_{G}(Q)$, $(1325)^2\in Q$ and $(1325)Q=\{(1325),(1523)\}$. Note that $(1325)Q$ contains no involution. By Lemma \ref{equivalent}, $Q$ is not a perfect code of $G$. It follows from Theorem \ref{sl2} that $H$ is not a perfect code of $G$. However, $Q$ is a perfect code of $P$ as $Q$ has a complement $\langle (12),(3456)\rangle$ in $P$.
\end{exam}

\section{Subgroup perfect codes of $\PSL(2,q)$}
\label{sec:psl}
Throughout this section, we assume that $q$ is a prime power. Let $d=1$ if $q$ is even and $d=2$ if $q$ is odd. Recall that $|\PSL(2,q)|=\frac{1}{d}q(q-1)(q+1)$.
Note that all subgroups of $\PSL(2,q)$ were first known by Dickson in \cite{D1901}. The main results of this section can be used to check whether a given subgroup of $\PSL(2,q)$ is perfect code. We will use the following result of Dickson which gives a classification of maximal subgroups of $\PSL(2,q)$.
\begin{lem}[\cite{D1901,Su}]
\label{max}
 A maximal subgroup of $\PSL(2,q)$ is isomorphic to one of the following groups:
\begin{enumerate}
  \item the dihedral group of order $\frac{2(q-1)}{d}$ when $q\neq 3,5,7,9,11$;
  \item the dihedral group of order $\frac{2(q+1)}{d}$ when $q\neq 2,7,9$;
  \item a semidirect product of an elementary abelian group of order $q$ by a cyclic group of order $\frac{q-1}{d}$;
  \item $S_4$ when $q$ is  an odd prime number and $q\equiv\pm1\pmod{8}$;
  \item $A_4$ when $q$ is a prime number $>3$ and $q\equiv3,13,27,37\pmod{40}$;
  \item $A_5$ when $q$ is one of the following forms: $q=5^m$ or $4^m$ where $m$ is a prime, $q$ is a prime number congruent to $\pm1\pmod 5$, or $q$ is the square of an odd prime number which satisfies $q\equiv-1\pmod 5$;
  \item $\PSL(2,r)$ when $q=r^m$ and $m$ is an odd prime number;
  \item $\PSL(2,r)$ when $q=r^2$.
\end{enumerate}
\end{lem}
The following lemma maybe well known and it can be deduced directly from Lemma \ref{max}.
\begin{lem}
\label{sypsl}
Let $q$ be a prime power with $q\equiv\pm1\pmod8$. Then every Sylow $2$-subgroup of $\PSL(2,q)$ is a dihedral group.
\end{lem}
\begin{proof}
Since $q\equiv\pm1\pmod8$, we have $|\PSL(2,q)|=\frac{1}{2}q(q-1)(q+1)$. Therefore $8\mid |\PSL(2,q)|$ and exactly one of $\frac{1}{2}q(q+1)$ and $\frac{1}{2}q(q-1)$ is odd. By Lemma \ref{max}, if $q\neq7,9$, then $\PSL(2,q)$ has dihedral subgroups of order $q\pm1$; if $q=7,9$, then $\PSL(2,q)$ has maximal subgroups isomorphic to $S_4$ (note that $\PGL(2,3)\cong S_4$). Thus $\PSL(2,q)$ has an dihedral Sylow $2$-subgroup and it follows from Lemma \ref{sylow} that every Sylow $2$-subgroup of $\PSL(2,q)$ is a dihedral group.
\end{proof}
In \cite{MWWZ20}, it was shown that every subgroup of $\PSL(2,q)$ is a perfect code if $q$ is even or $q\equiv\pm3\pmod8$. The following theorem  gives a criterion for determining subgroup perfect codes of $\PSL(2,q)$ for the remainder case when $q\equiv\pm1\pmod8$.
\begin{theorem}
\label{psl}
Let $q$ be a prime power with $q\equiv\pm1\pmod8$, $H$ a subgroup of $\PSL(2,q)$ and $Q$ a Sylow $2$-subgroup of $H$. Then $H$ is a perfect code of $\PSL(2,q)$ if and only if one of the followings holds:
\begin{enumerate}
  \item $Q$ is trivial;
  \item $Q$ is noncyclic;
  \item $Q$ is a cyclic $2$-group of maximal order.
\end{enumerate}
\end{theorem}
\begin{proof}
Write $G=\PSL(2,q)$. We at first prove the sufficiency. If $Q$ is trivial, then $H$ is of odd order. By Lemma \ref{odd}, $H$ is a perfect code of $G$. In what follows, we assume that $Q$ is nontrivial. By Lemma \ref{sypsl}, every Sylow $2$-subgroup of $G$ is dihedral. By Lemma \ref{sylow}, $Q$ is contained in a Sylow $2$-subgroup of $G$. Therefore, if $Q$ is noncyclic, then it is either a dihedral group or an elementary albean group of order $4$. Let $P$ be a Sylow $2$-subgroup of $N_{G}(Q)$. If $Q$ is noncyclic, then $P$ is a dihedral group and $Q$ is of index $2$ in $P$. If $Q$ is a cyclic $2$-group of maximal order, then $P$ is a Sylow $2$-subgroup of $G$ and therefore dihedral. In both cases, $Q$ has a complement of order $2$ in $P$ and therefore is a perfect code of $P$. By Lemma \ref{NGQ}, $H$ is a perfect code of $G$. This complete the proof of the sufficiency.

Now we prove the necessity. If $Q$ is a nontrivial proper subgroup of a cyclic $2$-group, then there exists $x\in G$ such that $x\notin Q$ and $Q=\langle x^2\rangle$. Clearly, $x\in N_{G}(Q)$ and $xQ$ contains no involution. By Lemma \ref{equivalent}, $Q$ is not a perfect code of $G$. It follows from Theorem \ref{sl2} that $H$ is not a perfect code of $G$. Thus, if $H$ is a perfect code of $G$, then one of the three conditions (i), (ii) and (iii) holds.
\end{proof}
It is straightforward to check that every maximal subgroup of $G$ listed in Lemma \ref{max} contains a Sylow $2$-subgroup which is either noncyclic or a cyclic $2$-group of maximal order. Combining the result of Ma el at \cite{MWWZ20} that every subgroup of $\PSL(2,q)$ is a perfect code if $q$ is even or $q\equiv\pm3\pmod8$,  Theorem \ref{psl} implies the following result.
\begin{coro}
Let $q$ be a prime power. Then every maximal subgroup of $\PSL(2,q)$ is a perfect code of $\PSL(2,q)$.
\end{coro}
\noindent {\textbf{Acknowledgements}}~~This work was supported by the Natural Science Foundation of Chongqing (CSTB2022NSCQ-MSX1054) and the Foundation of Chongqing Normal University (21XLB006).

\medskip
\noindent {\textbf{Data Availability} No data, models, or code were generated or used during the study.




\end{document}